\documentclass[12pt]{amsart}
\usepackage{somedefs,tracefnt,latexsym,,rawfonts,epsfig,amsfonts,amssymb,latexsym,enumerate,amscd}
\def\cal{\mathcal}
\def\Bbb{\mathbb}

\def\G{\Gamma}
\def\r{\rangle}
\def\l{\langle}
\def\t{\times}
\def\p{\partial}
\def\ul{\underline}
\def\flf{FICwF^{{\underline L}}}
\def\flv{FICwF^{{\underline L}}}

\def\ficwv{FICwF _{\cal {VC}}}

\def\flvu{FICwF^{{L}}}

\def\fpvu{FICwF^{{P}}}

\def\pvc{{\cal P} ^{L}_{\cal {VC}}}
\def\pvcu{{\cal P} ^{\ul L}_{\cal {VC}}}
\def\pfu{{\cal P} ^{\ul L}_{\cal {FIN}}}

\def\wtfinu{_w{\cal T} ^{\ul L}_{\cal {FIN}}}
\def\wttvu{_{wt}{\cal T} ^{\ul L}_{\cal {VC}}}
\def\wttv{_{wt}{\cal T}^{L}_{\cal {VC}}}
\def\wtfv{_{wf}{\cal T}^{L}_{\cal {VC}}}

\newtheorem{prop}{Proposition}[section]
\newtheorem{thm}{Theorem}[section]

\newtheorem{lemma}{Lemma}[section]
\newtheorem{cor}{Corollary}[section]

\newtheorem{defn}{Definition}[section]

\newtheorem{rem}{Remark}[section]
\numberwithin{equation}{section}
\begin{document}
\title[The isomorphism conjecture in $L$-theory]{The
 isomorphism conjecture in $L$-theory: graphs of groups}
\author[S.K. Roushon]{S. K. Roushon}
\address{School of Mathematics\\
Tata Institute\\
Homi Bhabha Road\\
Mumbai 400005, India}
\email{roushon@math.tifr.res.in} 
\urladdr{http://www.math.tifr.res.in/\~\ roushon/}
\thanks{April 06, 2010}
\begin{abstract}
We study the Fibered Isomorphism Conjecture of Farrell and Jones in $L$-theory 
for groups acting on trees. In several cases we prove the conjecture. This 
includes wreath products of abelian groups and free metabelian groups. We also 
deduce the conjecture in pseudoisotopy theory for these   
groups. Finally in $B$ of Theorem \ref{wreath} we prove the 
$L$-theory version of [\cite{FL}, Theorem 1.2].\end{abstract} 

\keywords{group action on trees, graph of groups, fibered 
isomorphism conjecture, $L$-theory, surgery groups}

\subjclass[2000]{Primary: 19G24, 19J25. Secondary: 55N91.}

\maketitle


\section{Introduction and statements of results} 
The classification problem for manifolds needs the study of two 
classes of obstruction groups. One is the lower $K$-groups 
(that is $K$-theory in dimension $\leq 1$)    
(pseudoisotopy-theory) and the other is the surgery 
$L$-groups (surgery theory) of the group ring of the 
fundamental group. 
The Farrell-Jones Isomorphism Conjecture gives an unified 
approach for computations and understanding of both these 
classes of groups. If this conjecture is true for  
the pseudoisotopy theory as well as for the surgery theory then 
among other results, for example, 
the Borel Conjecture, 
the Novikov Conjecture and the Hsiang Conjecture will 
be immediate consequences (see \cite{L}). The Farrell-Jones Conjecture 
predicts that one needs to consider only virtually cyclic 
subgroups of a group for computations of the above obstruction groups 
of the group. 

In this second article we are concerned about the Fibered 
Isomorphism Conjecture in surgery 
theory for groups acting (without inversion) on trees or equivalently 
for the fundamental groups of graphs of groups. 
Also we prove the conjecture for a certain class of 
virtually solvable groups 
in both the pseudoisotopy and surgery theory. 
The Fibered Conjecture in surgery theory for various 
classes of groups were proved in \cite{R3}. Also some machinery was set up 
in \cite{R3} which are crucial in this paper.

The Fibered Isomorphism Conjecture is stronger  
and has hereditary property. 
Also it 
allows one to consider groups 
with torsion in induction steps, although the final 
aim is to prove results for torsion free groups. This 
technique was first used in \cite{FR} to prove the 
conjecture in the pseudoisotopy case 
for Artin full braid groups. The general methods 
in \cite{R1} and \cite{R3} extend this feature 
further by considering the conjecture always for 
groups wreath product with finite groups. This 
simplifies proofs and   
prove stronger results.

In most of our results 
of the Fibered Isomorphism Conjecture in the equivariant 
homology theory (\cite{BL}) we need the assumption 
that $_{wt}{\cal T} _{\cal {VC}}$,  
${\cal P} _{\cal {VC}}$ and 
${\cal L} _{\cal {VC}}$ 
(see Definition \ref{property}) are satisfied.  
We checked before that these conditions are satisfied 
for the $L^{\l -\infty \r}$ and for 
the pseudoisotopy version of the conjecture. (See \cite{R1} 
and \cite{R3}). In [\cite{BLR}, Theorem 0.1] 
it is included that ${\cal L} _{\cal {VC}}$ and 
${\cal P} _{\cal {VC}}$ are  
satisfied for the $K$-theory case 
of the conjecture.

Formally, the conjecture in surgery theory 
says that certain assembly map in $L^{\l -\infty \r}$-theory  
is an isomorphism. A weaker version of the conjecture 
is that the assembly map is an isomorphism 
after tensoring with ${\Bbb Z}[\frac{1}{2}]$. This eliminates 
the UNiL groups of Cappell and the tensored assembly map 
can be proven to be 
an isomorphism for a larger class of groups. In addition to some 
general results we also prove 
the isomorphism of this tensored assembly map for 
a large class of groups acting on trees.

For two groups $G$ and $H$, $G\wr H$   
denotes the (restricted) wreath product with respect 
to the regular action of $H$ on $G^H$. By definition 
$G^H=\bigoplus_{h\in H} G_h$ where $G_h$'s 
are copies of $G$ indexed by $H$. And the action of 
$H$ on $G^H$ is such that $h'\in H$ sends an element 
of $G_h$ to the corresponding element of $G_{h(h')^{-1}}$.

If the Fibered Isomorphism 
Conjecture is true for $G\wr F$ for all finite groups $F$ for the 
$L^{\l -\infty \r}$, 
${\ul L}^{\l -\infty \r}=
L^{\l -\infty \r}\otimes_{\Bbb Z}{\Bbb Z}[\frac{1}{2}]$ or for the 
pseudoisotopy theory then we say respectively that 
the $\flvu$, $\flv$ or $\fpvu$ is true for $G$. 

Throughout the article a `graph' is assumed to be connected 
and locally finite. And groups are assumed to be discrete and 
countable.

\begin{defn}\label{closely}{\rm A finitely generated 
group $G$ is called 
{\it closely crystallographic} if it is of the form 
$A\rtimes C$ where $A$ is torsion free abelian, $C$ is infinite 
cyclic and $A$ is irreducible as a ${\Bbb Q}[C]$-module.}\end{defn}

When $C$ is virtually cyclic then $G$ was defined as 
{\it nearly crystallographic} in [\cite{FL}, Definition].

Our first theorem is the following.

\begin{thm} \label{wreath} 
A. Let $G$ be a group which contains a subgroup 
$H$ of finite index so that $H$ belongs to 
one of the following classes.

a. $A\wr B$ where $A$ and $B$ are both abelian.

b. Free metabelian groups. That is, it is a
quotient of a free group by the second derived subgroup.

c. $A\rtimes {\Bbb Z}$ where $A$ is torsion abelian.

Then the $\flvu$ and the $\fpvu$ are satisfied for $G$.

B. If the $\flvu$ ($\fpvu$) is true for all closely crystallographic 
groups then the $\flvu$ ($\fpvu$) is true for all virtually 
solvable groups.\end{thm}

\begin{rem}\label{miller} {\rm It is not yet known if the 
Fibered Isomorphism 
Conjecture is true for all metabelian groups. The simplest case 
for which it is unknown is ${\Bbb Z}[\frac{1}{2}]\rtimes {\Bbb Z}$ 
where the action of 
$\Bbb Z$ on ${\Bbb Z}[\frac{1}{2}]$ is multiplication by $2$. One 
can show that ${\Bbb Z}[\frac{1}{2}]\rtimes {\Bbb Z}$ can not be 
embedded in $A\wr B$ where $A$ and $B$ are both abelian. I thank 
Chuck Miller for explaining this fact to me. On the other hand 
by a result of Magnus free metabelian groups can be embedded in 
such a wreath product.}\end{rem}

Although our method does not work to deduce the Fibered Isomorphism 
Conjecture 
in the closely crystallographic case,  
the Isomorphism Conjecture can be proven for these groups in surgery 
theory for all the decorations. 

\begin{thm}\label{ic} The Isomorphism Conjecture in $L^i$-theory 
is true for closely 
crystallographic groups where $i=\l -\infty \r, h$ or 
$s$.\end{thm}

\begin{rem} {\rm Here we recall that in [\cite{FL}, Theorem 1.2] it was 
proved that the Fibered Isomorphism Conjecture in the pseudoisotopy 
theory is true for any 
virtually solvable groups if the same is true for any nearly 
crystallographic groups. Thus $B$ of Theorem \ref{wreath} 
is the $L$-theory version of [\cite{FL}, Theorem 1.2].}\end{rem}

The following is an Important Assertion 
in the Fibered Isomorphism Conjecture. In general 
it is not yet known. 

{\bf IA(K)}. $K$ is a normal subgroup of a group $G$ with 
infinite cyclic quotient. If the $\ficwv(K)$ is satisfied 
then the $\ficwv(G)$ is also satisfied.

\begin{thm}\label{residually1} Let $\cal G$ be a graph of 
groups with finite edge groups.

A. If the vertex groups are residually finite and 
the $\flvu$ is true for the vertex groups of $\cal G$ then 
the $\flvu$ is true for $\pi_1({\cal G})$.

B. Assume that there is a homomorphism $f:\pi_1({\cal G})\to Q$. Then 
the following statements hold.

i. If the kernels of the restriction of $f$ to the vertex groups
of $\cal G$ are finitely generated, residually finite and 
satisfies the $\flvu$ then
the $\flvu$ is true for $\pi_1({\cal G})$ provided the same 
is true for $Q$ and the IA(V) is satisfied for all vertex groups 
$V$ of $\cal G$  
in the $L$-theory case.

ii. If the kernels of the restriction of $f$ to the vertex groups 
of $\cal G$ 
are virtually polycyclic and the $\flvu$ is true for $Q$ then the 
$\flvu$ is true for $\pi_1({\cal G})$.\end{thm}

Let us now recall from \cite{R1} the following definitions. 
$V_{\cal G}$ and $E_{\cal G}$ denotes 
respectively the set of all vertices and edges of a 
graph of groups $\cal G$.  
${\cal G}_x$ denotes a vertex or 
an edge group for $x$ a vertex or an edge respectively.
An edge $e$ of $\cal G$ is called a {\it finite edge}
if the edge group ${\cal G}_e$ is finite. $\cal G$ is called
{\it almost a tree
of groups} if there are finite edges $e_1,e_2,\ldots$ so that the
components of ${\cal
G}-\{e_1,e_2,\ldots \}$ are tree. If we remove all the finite edges
from a graph of groups then we call the components of the resulting
graph as {\it component subgraphs}.
A graph of groups $\cal G$ is 
said to satisfy the {\it intersection property} if each 
connected subgraph of groups
${\cal G}'$ of $\cal G$, $\cap_{e\in E_{{\cal G}'}}{\cal G}'_e$
contains a subgroup which is normal in 
$\pi_1({\cal G}')$ and is of finite index in some edge group.
A group $G$ is called {\it subgroup separable} if
for any finitely generated subgroup $H$ of $G$ and for any $g\in G-H$
there is a finite index normal subgroup $N$ of $G$ so that $H\subset N$
and $g\in G-N$.

\begin{thm} \label{graph-group} 
The $\flvu$ is true for $\pi_1({\cal G})$ where $\cal G$ 
satisfies one of the following.

A. $\cal G$ is a graph of poly-cyclic groups with intersection 
property. 

B. $\cal G$ is a graph of finitely generated nilpotent groups 
with $\pi_1({\cal G})$ subgroup separable.

C. The vertex groups are virtually cyclic and any component 
subgraph is either a single vertex or a tree of abelian groups.

D. The 
vertex and edge groups  
of any component subgraph are finitely generated abelian and 
of the same rank and any component subgraph is a tree.\end{thm}

Finally we state our results in the ${\ul L}^{\l -\infty \r}$-theory 
case.

Let $\cal D$ be a class of groups which is closed 
under isomorphism. For a graph $\cal G$ 
we denote by 
${\cal D}_{\cal G}$ the class of graphs of groups whose vertex 
and edge groups belong to $\cal D$ and the underlying graph 
is $\cal G$.

\begin{thm}\label{reduction1}
A. If the $\flf (\pi_1({\cal T}))$ 
is satisfied for all tree of groups $\cal T$ then 
the $\flf (\pi_1({\cal G}))$
is satisfied for all graph of groups $\cal G$.

B. If the $\flf (\pi_1({\cal H}))$ 
is satisfied for all ${\cal H}\in {\cal D}_{\cal T}$  
and for all tree $\cal T$ then the 
$\flf (\pi_1 ({\cal H}))$ is satisfied  
for all ${\cal H}\in {\cal D}_{\cal G}$ and 
for all graph $\cal G$.\end{thm}

\begin{thm} \label{abelian} Let $\cal G$ be a graph of 
finitely generated abelian 
groups. Then the $\flf$ is true for $\pi_1({\cal G})$.\end{thm}

\begin{thm} \label{residually2} 
Let $\cal G$ be a graph of groups with finite 
edge groups.

A. If the vertex groups are residually finite and 
the $\flf$ is true for the vertex groups of $\cal G$ then 
the $\flf$ is true for $\pi_1({\cal G})$.

B. Assume that there is a homomorphism $f:\pi_1({\cal G})\to Q$ and the 
$\flf$ is true for $Q$. 
If the kernels of the restriction of $f$ to the vertex groups 
of $\cal G$ are residually finite and satisfies the $\flf$ then 
the $\flf$ is true for $\pi_1({\cal G})$.\end{thm}

\section{Statement of the Isomorphism Conjecture and some basic results}

Now we proceed to describe the formal statement of the conjecture (see 
\cite{BL}) and 
introduce some notations.

Let ${\cal H}^?_*$ be an equivariant homology theory with values in 
$R$-modules for $R$ a commutative associative ring with unit. In 
this article we are considering the special case $R={\Bbb Z}$.

In this section we always assume that a class of 
groups $\cal C$ is closed under 
isomorphisms, taking subgroups and taking quotients. We 
denote by ${\cal C}(G)$ the set of all subgroups of a group 
$G$ which belong to $\cal C$. In this case ${\cal C}(G)$ 
is said to be a {\it family of subgroups} of $G$. It follows that 
${\cal C}(G)$ is closed under taking subgroup and conjugation.
   
Given a group homomorphism $\phi:G\to H$ and $\cal C$ a family of 
subgroups of $H$ define $\phi^*{\cal C}$ by the family  
of subgroups $\{K<G\ |\ \phi (K)\in {\cal C}\}$ of $G$. For a family   
$\cal C$ of subgroups of a group $G$ there is a $G$-CW complex $E_{\cal 
C}(G)$ which is unique up to $G$-equivalence satisfying the property that 
for each $H\in {\cal C}$ the fixpoint set $E_{\cal C}(G)^H$ is 
contractible and $E_{\cal C}(G)^H=\emptyset$ for $H$ not in ${\cal C}$. 

The {\it 
Isomorphism Conjecture} for the pair $(G, {\cal C})$ states that the 
projection 
$p:E_{\cal C}(G)\to pt$ to the point $pt$ induces an isomorphism $${\cal 
H}^G_n(p):{\cal H}^G_n(E_{\cal C}(G))\simeq {\cal H}^G_n(pt)$$ for $n\in 
{\Bbb Z}$. 

And the {\it Fibered Isomorphism Conjecture} for the pair $(G, {\cal 
C})$ states that for any group homomorphism $\phi: K\to G$ the 
Isomorphism Conjecture is true for the pair $(K, \phi^*{\cal C})$.

\begin{defn} \label{definition} ([\cite{R1}, Definition 2.1]) 
{\rm Let $\cal C$ be a class of groups. If the (Fibered) 
Isomorphism Conjecture is
true for the pair $(G, {\cal C}(G))$ we say that the {\it (F)IC$_{\cal C}$ is true for $G$} or simply say {\it (F)IC$_{\cal C}(G)$ is satisfied}. Also we say that the {\it 
(F)ICwF$_{\cal C}(G)$ is satisfied} if
the {\it (F)IC$_{\cal C}$} is true for $G\wr H$ for any finite group
$H$.}\end{defn} 

Clearly, if $H\in {\cal C}$ then the (F)IC$_{\cal C}(H)$ is satisfied. 

Let us denote by $P$, ${\ul L}$ and 
${L}$, 
the equivariant homology 
theories arise for the pseudoisotopy theory, $\ul L^{\l -\infty \r}$-theory 
and for the $L^{\l -\infty \r}$-theory respectively. We also 
denote the corresponding conjectures with respect to the 
class of groups $\cal{VC}$ by (F)IC$^X$  
((F)ICwF$^X$) where $X=P, {\ul L}$ or $L$.

\begin{defn} \label{property} ([\cite{R1}, Definition 2.2]) 
{\rm We say that $_{wt}{\cal T} _{\cal C}$ 
is satisfied if for a graph of groups $\cal G$ with trivial edge
groups and the vertex
groups belonging to the class $\cal C$, the FICwF$ _{\cal
C}$ for $\pi_1({\cal G})$ is true . 

And we say that {\it ${\cal P} _{\cal C}$} is 
satisfied if for $G_1, G_2\in {\cal C}$ the product
$G_1\times G_2$ satisfies the FIC$ _{\cal C}$.

We further say that ${\cal L} _{\cal C}$ 
is satisfied if for any 
directed sequence of groups $\{G\}_{i\in I}$ for which 
the FIC$ _{\cal C}(G_i)$ is satisfied 
for $i\in I$ then 
the FIC$ _{\cal C}(\lim_{i\in I}G_i)$ 
is satisfied.

We denote the above properties for the equivariant 
homology theories $P$, $\ul L$ and $L$ with a super-script 
by the corresponding theory. For example ${\cal P}_{\cal C}$ 
for $L$ is denoted by ${\cal P}^L_{\cal C}$.}\end{defn}

We now recall some results we need to prove the Theorems.

\begin{lemma} \label{limit} Assume that 
${\cal L} _{\cal C}$ 
is satisfied. If for a  
directed sequence of groups $\{G\}_{i\in I}$  
the FICwF$ _{\cal C}(G_i)$ is satisfied 
for $i\in I$ then 
the FICwF$ _{\cal C}$ is true for $\lim_{i\in I}G_i$.\end{lemma}

\begin{proof} Given a finite group $F$ note the following 
equality. $$(\lim_{i\in I}(G_i))\wr F=\lim_{i\in I}(G_i\wr F).$$
The proof now follows.\end{proof}

The following is easy to prove and is known as 
the hereditary property of the Fibered Isomorphism 
Conjecture.

\begin{lemma} \label{heredi} If the FIC$ _{\cal C}$ 
(FICwF$ _{\cal 
C}$) is true for a group $G$ then the 
FIC$ _{\cal C}$ 
(FICwF$ _{\cal C}$) 
is true for any subgroup $H$ of $G$.\end{lemma}

\begin{lemma}\label{inverse} ([\cite{R3}, Lemma 2.2]) 
Assume that ${\cal P}_{\cal C}$ is satisfied. 

(1). If $G_1$ and $G_2$ satisfy 
the FIC$ _{\cal C}$ (FICwF$_{\cal C}$) then $G_1\times G_2$ 
satisfies the FIC$ _{\cal C}$ (
FICwF$_{\cal C}$). 

(2). Let $G$ be a finite index subgroup of a group $K$. If  
the group $G$ satisfies the FICwF$_{\cal C}$ then $K$ also satisfies the 
FICwF$ _{\cal C}$.

(3). Let $p:G\to Q$ be a group homomorphism. If the FICwF$_{\cal C}$ is true for $Q$ and for $p^{-1}(H)$ for all 
$H\in {\cal C}(Q)$ then the FICwF$_{\cal C}$ is true for $G$. If ${\cal C}={\cal {VC}}$ then 
using $(2)$ it is enough to consider 
$H\in {\cal C}(Q)$ to be infinite cyclic.\end{lemma}

\begin{lemma}\label{pvc} ([\cite{R1}, Corollary 5.3] and [\cite{R3}, Lemma 2.11]) 
The properties ${\cal P} ^{P}_{\cal {VC}}$, $\pvc$, $\pvcu$ and  $\pfu$ are 
satisfied.\end{lemma}

\begin{lemma}\label{tvc} ([\cite{R1}, Corollary 2.1] and [\cite{R3}, Lemma 2.14]) 
The properties $_{wt}{\cal T} ^{P}_{\cal {VC}}$, $\wtfinu$, $\wttv$ and $\wttvu$ are
satisfied.\end{lemma}

The proofs of the properties $\cal P$ and $\cal T$ in the  
$L$-theory case were given in \cite{R3} as referred using 
[\cite{FJ}, Theorem 2.1 and Remark 2.1.3]. See Remark 
\ref{finalrem} regarding the present 
status of the proof of [\cite{FJ}, Theorem 2.1 and Remark 2.1.3]. 
Here we sketch alternate proofs of 
the above properties using some recent result of 
Bartels and L\"{u}ck in \cite{BL1}. In fact  
we can even prove $\wtfv$. 

\begin{proof}[Alternate proofs of $\pvc$, $\wtfinu$ and $\wttv$] The proof 
of these facts in the pseudoisotopy case of the Fibered Isomorphism 
Conjecture were given in \cite{R1}. The same proofs also apply  
in the $L$-theory case if we use \cite{BL1}. We describe below the 
changes required.
 
For $\pvc$ replace $P$ by $L$ and use 
[\cite{BL1}, Theorem B] in the proof of [\cite{R1}, Corollary 5.3]. Also 
see [\cite{R4}, Section 3] for some more on this matter.

For a proof of $\wttv$ use the last paragraph of the proof 
of [\cite{R1}, Proposition 2.4] after replacing $P$ by $L$ 
and use [\cite{BL1}, Theorem B]. We also need to use some 
basic deductions from [\cite{R4}, Section 3].
The proof of the second property is immediate.\end{proof}

For the proof of $\wtfv$ we again use the proof of  
[\cite{R1}, Proposition 2.4].

The following Lemma is another ingredient for 
the proofs of the Theorems.

\begin{lemma}\label{lvc} ([\cite{FL}, Theorem 7.1]) 
The properties ${\cal L} ^{P}_{\cal {VC}}$, 
${\cal L} ^{L}_{\cal {VC}}$,
${\cal L} ^{\ul L}_{\cal {VC}}$ and 
${\cal L} ^{\ul L}_{\cal {FIN}}$ are satisfied.\end{lemma}

Finally we recall the following two lemmas.

\begin{lemma} \label{claim} ([\cite{R1}, Lemma 6.3]) 
Assume ${\cal P} _{\cal
C}$ and $_{wt}{\cal T} _{\cal
C}$ are
satisfied. If the FICwF$ _{\cal C}$ is true
for
$G_1$ and $G_2$ then 
the FICwF$ _{\cal C}$ 
is
true for $G_1 * G_2$.\end{lemma}

\begin{lemma} \label{abelian1} Let ${\cal {FIN}}\subset {\cal C}$. 
Assume ${\cal P} _{\cal {C}}$, 
${\cal L} _{\cal {C}}$ and 
$_{wt}{\cal T} _{\cal {C}}$ are satisfied. Then the 
FICwF$ _{\cal C}$ is true for $G$ where $G$ is either 
a virtually abelian group or a virtually free group.\end{lemma} 

\begin{proof}
By Lemma \ref{limit} we can assume that the group $G$ 
is finitely generated. 
At first assume that $G$ is virtually abelian. 
Using $(2)$ of Lemma \ref{inverse} 
we reduce to the case of finitely generated abelian groups. Since 
${\cal {FIN}}\subset \cal C$ and since the conjecture is true 
for members of $\cal C$ it is enough to prove the Lemma 
for finitely generated 
free abelian groups. Now $(1)$ of Lemma \ref{inverse} implies that 
we need to consider only the infinite cyclic group. Since the 
fundamental group of a graph of groups with trivial stabilizers 
is a free group, we are done using 
$_{wt}{\cal T} _{\cal C}$ and Lemma \ref{heredi}.

Next assume that $G$ is finitely generated and virtually free. Again 
using $(2)$ of Lemma \ref{inverse} it is enough to assume that $G$ is a  
finitely generated free groups. Now note that a free group is 
isomorphic to the fundamental group of a graph of groups 
whose vertex groups are trivial. Hence using $_{wt}{\cal T} _{\cal {C}}$ 
we complete the proof.\end{proof}

\section{Proofs of the Theorems} 
For the proof of Theorem \ref{wreath} we prove the following 
general Theorem for the conjecture 
in equivariant homology theory. The advantage of this   
general statement is that it works for the conjecture in 
any equivariant homology theory and to prove 
Theorem \ref{wreath} we just have to show that the hypotheses 
are satisfied both for the pseudoisotopy and 
for the $L^{\l-\infty\r}$-theory case.

\begin{thm}\label{general} Assume that 
$_{wt}{\cal T} _{\cal {VC}}$, 
${\cal P} _{\cal {VC}}$ and  
${\cal L} _{\cal {VC}}$ are 
satisfied. Then the following hold.

1. The $\ficwv$ is true for $G$ if $G$ contains $A\wr B$ 
as a subgroup of finite index, where $A$ and $B$  
are abelian groups.

2. The $\ficwv$ is true for any virtually free 
metabelian group.

3. The $\ficwv$ is true for $A\rtimes {\Bbb Z}$ where $A$ is 
torsion abelian.

4. The $\ficwv$ is true for any virtually solvable groups 
provided it is true for any closely crystallographic 
groups.\end{thm}

\begin{rem}{\rm In the algebraic $K$-theory version of 
the conjecture ${\cal P} _{\cal {VC}}$ and 
${\cal L} _{\cal {VC}}$ are known. See 
[\cite{BLR}, Theorem 0.1]. But it is not yet known if 
$_{wt}{\cal T} _{\cal {VC}}$ is also satisfied. 
If this is the case then together with the results in this paper 
most of the results from \cite{R1} and \cite{R3} will 
be true for the Fibered Isomorphism Conjecture in 
algebraic $K$-theory.}\end{rem}

Here we should recall that it was proved in 
[\cite{FL}, Lemma 4.3] that 
the pseudoisotopy version of the Fibered Isomorphism 
Conjecture is true for $({\Bbb Z}^n\wr {\Bbb Z})\wr F$, where 
$F$ is a finite group. That is, 
the $FICwF ^{P}$ is true for 
${\Bbb Z}^n\wr {\Bbb Z}$. 

\begin{proof}[Proof of Theorem \ref{general}]
Using $(2)$ of Lemma \ref{inverse} it is enough to prove the $\ficwv$ 
for $A\wr B$ where $B$ is infinite, for free metabelian groups 
and for solvable groups 
under the respective hypotheses as in $(1)$, $(2)$ and $(3)$. 
Also we will use the fact that the $\ficwv$ is true 
for virtually abelian groups during the proof. See 
Lemma \ref{abelian1}.

{\bf Proof of 1.} At first we reduce the situation to 
the case $A\wr {\Bbb Z}$. Let $B=\lim_{j\in J}B_i$ where $B_j$ are 
increasing sequence of finitely generated subgroups of $B$. 
Then we get the following 
equality. $$A\wr B=\lim_{j\in J}(A\wr B_j).$$ Therefore 
from now on we can assume that  $B$ is finitely generated. 
If $B$ has rank equal to $k$ then $B={\Bbb Z}^k\t F_1$ where $F_1$ is finite.
Hence $A\wr B$ contains 
$A^{{\Bbb Z}^k\t F_1}\rtimes {\Bbb Z}^k=A^{F_1}\wr {\Bbb Z}^k$ as a subgroup of 
finite index. Therefore we can use $(2)$ of Lemma \ref{inverse} to reduce the 
situation to the case $A\wr {\Bbb Z}^k$. If $k\geq 2$ then  
note the following equality. Let ${\Bbb Z}^k=B_1\t B_2$ where 
$B_1$ and $B_2$ are both nontrivial. Then 
$$A\wr (B_1\t B_2)=A^{B_1\t B_2}\rtimes (B_1\t B_2)
<(A^{B_1\t B_2}\rtimes B_1)\t (A^{B_1\t B_2}\rtimes B_2)$$$$
\simeq (A^{B_2}\wr B_1)\t (A^{B_1}\wr B_2).$$ 

In the above display, for $i=1,2$, the action of $B_i$ 
on $A^{B_1\t B_2}$ is 
the restriction of the regular action of $B_1\t B_2$ 
on $A^{B_1\t B_2}$. Note that the restricted action of $B_1$ ($B_2$) 
is again regular on $(A^{B_2})^{B_1}$ ($(A^{B_1})^{B_2}$). 
And the second inequality is easily 
checked by showing that 
the map $$A^{B_1\t B_2}\rtimes (B_1\t B_2)
\to (A^{B_1\t B_2}\rtimes B_1)\t (A^{B_1\t B_2}\rtimes B_2)$$ 
defined by $$(x, (b_1, b_2))\mapsto ((x, b_1), (x, b_2))$$
for $x\in A^{B_1\t B_2}$ and $(b_1, b_2)\in B_1\t B_2$ 
is an injective homomorphism.

Therefore using $(1)$ of Lemma \ref{inverse} and by the hereditary 
property it is enough to 
prove the $\ficwv$ for groups of the form $A\wr {\Bbb Z}$ where $A$ 
is abelian.

Since $A$ is countable abelian we can write it as a limit of finitely 
generated abelian subgroups $A_i$. Now note that 
$A\wr {\Bbb Z}=(\lim_{i\in I}A_i)\wr {\Bbb Z}
=\lim_{i\in I}(A_i\wr {\Bbb Z})$. Hence 
by Lemma \ref{limit} it is 
enough to prove the $\ficwv$ for $A_i\wr {\Bbb Z}$. 

Therefore from now on we can assume that $A$ is finitely 
generated.

Next note the equality in the following Lemma. 
This was obtained in the 
proof of [\cite{FL}, Lemma 4.3]. 

\begin{lemma}\label{far-lin} Let $A$ be an abelian group. Then 
the following equality holds.

$$A\wr {\Bbb Z}=
\lim_{n\to\infty} (A^{n+1}*_{A^n}).$$ Where the HNN extension 
$A^{n+1}*_{A^n}=H_n$ (say) is obtained using the following two 
inclusions. $$i_j:A^n\to A^{n+1}.$$ 
$$i_1(a_1,\ldots , a_n)\mapsto (a_1,\ldots , a_n, 0).$$
$$i_2(a_1,\ldots , a_n)\mapsto (0, a_1,\ldots , a_n).$$\end{lemma}

Again by Lemma \ref{limit} we 
need to prove the $\ficwv$ for $H_n$.

We have a surjective homomorphism 
$p:H_n\to A^{n+1}\rtimes_{\alpha}{\Bbb Z}=\ul H_n$ (say) where 
$\alpha(a_1,\ldots , a_{n+1})=(a_{n+1}, a_1,\ldots , a_n)$.

Recall that $A$ is a finitely generated abelian group. 
Let $B$ be a finitely generated free abelian 
subgroup of $A$ of finite index. Clearly $\alpha$ 
leaves $B^{n+1}$ invariant. 
Therefore $B^{n+1}\rtimes_{\alpha}{\Bbb Z}=\ul G_n$ (say)  is a 
finite index subgroup 
of $\ul H_n$. Hence $p^{-1}(\ul G_n)$ is a 
finite index subgroup of $H_n$. Obviously 
$p^{-1}(\ul G_n)=B^{n+1}*_{B^n}=G_n$ (say). Where the HNN 
extension $G_n$ is obtained by the same maps $i_j$ as we 
defined above.

We now use $(2)$ of Lemma \ref{inverse} to reduce
the situation to $G_n$. That is we need to prove the 
$\ficwv$ for $G_n$. We would like to apply $(3)$ 
of Lemma \ref{inverse} to $p:G_n\to \ul G_n$. 
From now on we follow the proof of Lemma 4.3 in page 314 
in \cite{FL}. 

Let $C$ be a virtually cyclic 
subgroup of $\ul G_n$. Since $\ul G_n$ is torsion free 
$C$ is either trivial or infinite cyclic. 

Since $G_n$ is 
an $HNN$-extension it acts on a tree with vertex stabilizer 
conjugates of $B^{n+1}$ and edge stabilizers conjugates of 
$B^n$. Therefore ker$(p)$ also acts on this tree and it follows 
that the stabilizers of this restricted action are trivial. 
Hence ker$(p)$ is a free group by [\cite{R1}, Lemma 3.2].

When $C$ is infinite cyclic then in the proof 
of [\cite{FL}, Lemma 4.3] (see paragraphs 
2 and 3 in page 316 in \cite{FL})  
it was deduced that $p^{-1}(C)$ is a direct limit of 
finitely generated subgroups 
$C_i$ (say) so that each $C_i$ is a subgroup of a finite 
free product $K*\cdots *K$ where $K$ is isomorphic to a 
direct product of a finitely generated free group and 
an infinite cyclic group. 

Now since $_{wt}{\cal T} _{\cal {VC}}$ 
is satisfied 
the $\ficwv$ is true for free groups. 
See Lemma \ref{abelian1}. Therefore the $\ficwv$ 
is true for ker$(p)$. Also by Lemma \ref{claim} the  
$\ficwv$ is true for the free product of two groups if the 
$\ficwv$ is true for each free summand and  
$_{wt}{\cal T} _{\cal {VC}}$ and 
${\cal P} _{\cal {VC}}$ are satisfied. Therefore, in 
addition, 
using $(1)$ of Lemma \ref{inverse} we deduce that the $\ficwv$ 
is true for $K*\cdots *K$ and hence for $C_i$ also 
by Lemma \ref{heredi}. Finally 
by Lemma \ref{limit} we conclude that the 
$\ficwv$ is true for $p^{-1}(C)$.

Therefore $G_n$ satisfies the $\ficwv$ for each $n$.

This completes the proof of $(1)$.

{\bf Proof of 2.} 
Let $G$ be a free metabelian group. Then the Magnus Embedding 
Theorem (\cite{Mg}) says that $G$ can be 
embedded as a subgroup of a group of the form $A\wr B$ where 
$A$ and $B$ are abelian. The proof of $(2)$ now follows from $(1)$ 
using Lemma \ref{heredi}.

{\bf Proof of 3.} The proof follows the steps of the 
proof of [\cite{FL}, Corollary 4.2].

Using Lemma \ref{limit} we assume that 
$G=A\wr {\Bbb Z}$ is finitely generated. This makes $A$ a 
finitely generated ${\Bbb Z}[{\Bbb Z}]$-module via the 
conjugation action of $G$ on $A$. Hence $A$ has finite exponent. 

Let us first assume that we have proved the result when 
this exponent is a prime. To complete the proof we now 
use induction on the exponent, say $\tau$. If $\tau=1$ 
then there is nothing to prove. So assume $\tau=pq\geq 2$ and 
$p$ is a prime. Note that $pA$ is a normal subgroup of $G$ and 
hence we have the following two exact sequences. 
$$1\to pA\to G\to G_1\to 1$$ $$1\to A/pA\to G_1\to {\Bbb Z}\to 1.$$

Note that the exponent of $A/pA$ is $p$ and hence the $\ficwv$ is true 
for $G_1$ by assumption. Next, the exponent of $pA$ is $q< \tau$ 
and hence by the induction hypothesis and applying 
$(3)$ of Lemma \ref{inverse} to the homomorphism 
$G\to G_1$ we are done.

Let us now assume that the exponent of $A$ 
is a prime $p$ and complete the proof. This makes  
$A$ a finitely generated ${\Bbb Z}_p[{\Bbb Z}]$-module. 
Since ${\Bbb Z}_p[{\Bbb Z}]$ is a PID $A$ has a decomposition 
in free part and torsion part as ${\Bbb Z}_p[{\Bbb Z}]$-module. 
Let $A_0$ be the free part. Then $A_0$ is a normal subgroup of 
$G$. Let $C$ be an infinite cyclic subgroup of $G$ which 
goes onto $\Bbb Z$ under the map $G\to {\Bbb Z}$. Then 
$A_0C$ is a finite index subgroup of $G$ and also 
$A_0C\simeq {\Bbb Z}^n_p\wr {\Bbb Z}$ where $n$ is the 
rank of $A_0$ as a free ${\Bbb Z}_p[{\Bbb Z}]$-module.
Hence using $(2)$ of Lemma \ref{inverse} we are done 
once we show that the $\ficwv$ is true for ${\Bbb Z}^n_p\wr {\Bbb Z}$.

Let $B={\Bbb Z}^n_p$ then by Lemma \ref{far-lin} we have the following 
equality. $B\wr {\Bbb Z}\simeq \lim_{k\to \infty}B^{k+1}*_{B^k}$. Next 
note that $B^{k+1}*_{B^k}$ is finitely generated and isomorphic to 
the fundamental group of a graph of 
finite groups and hence contains a free subgroup 
of finite index (see [\cite{R1}, Lemma 3.2]. 
Finally using Lemma \ref{abelian} we complete the 
proof.

{\bf Proof of 4.}
The proof uses the method of the proof of 
[\cite{FL}, Corollary 4.4]. 

For a solvable group $G$ we say that it is $n$-{\it step} solvable 
if $G^{(n)}=(1)$ and $G^{(n-1)}\neq (1)$, where $G^{(i)}$ denotes 
the $i$-th derived subgroup of $G$. 

Let $G$ be an $n$-step solvable group. 
The proof of $(3)$ is by induction on $n$. So assume that if 
the $\ficwv$ is true for all closely crystallographic groups 
then it is true for all $k$-step solvable groups for 
$k\leq n-1$. 

We have an exact sequence $1\to G^{(2)}\to G\to G/G^{(2)}\to 1$. 
By $(3)$ of Lemma \ref{inverse} and by the induction hypothesis 
it is enough to prove the $\ficwv$ for $2$-step solvable groups, since 
for any infinite cyclic subgroup of $G/G^{(2)}$ the inverse image under 
the quotient map $G\to G/G^{(2)}$ is an $(n-1)$-step solvable group.

Therefore we have reduced the proof 
to the following situation. 

$$1\to G^{(1)}\to G\to G/G^{(1)}\to 1.$$

Here $G^{(1)}$ and $G/G^{(1)}$ are both abelian. 

By Lemma \ref{abelian1} 
we can assume that $G/G^{(1)}$ is infinite. Again 
applying $(3)$ of Lemma \ref{inverse} to the map $G\to G/G^{(1)}$ we 
see that it is enough to prove the $\ficwv$ for the 
group $G=A\rtimes {\Bbb Z}$ where $A$ is an abelian group.
 
Let $A_T$ be the subgroup of $A$ consisting of all 
elements of finite order. Then $A_T$ is a 
characteristic subgroup of $A$ and hence we have an 
exact sequence. $$1\to A_T\to G\to G/A_T\to 1.$$ 
Note that $G/A_T\simeq (A/A_T)\rtimes {\Bbb Z}$.

Therefore by $(2)$ and $(3)$ of Lemma \ref{inverse} 
it is enough to 
prove the $\ficwv$ for $G$ for the following two 
individual cases.
 
$Case (a).$ $A$ is torsion abelian. 

$Case (b).$ $A$ is torsion free abelian. 

Proof of $Case (a):$ This case is same as $(2)$.

Proof of $Case (b):$
Note that by Lemma \ref{limit} we may assume that 
$A$ is finitely generated as a 
${\Bbb Q}[{\Bbb Z}]$-module (see the proof of 
[\cite{FL}, Corollary 4.4]). As ${\Bbb Q}[{\Bbb Z}]$ is a 
principal ideal domain, $A\simeq X\oplus Y$ where 
$X$ is the sum
of free and $Y$ is sum of finite ${\Bbb Q}$-dimensional 
${\Bbb Q}[{\Bbb Z}]$-submodule of $A$. Let $m=$ dim $Y$ and 
$n$ is the number of free parts in $X$. Note that $Y$ is a 
normal subgroup of $G$. The proof is now 
by induction first on $n$ and then on $m$. If 
$m=n=0$ then $G$ is infinite cyclic 
so there is nothing to 
prove. So assume $n=0$ and $m>0$. Let $Y_0$ be an 
irreducible ${\Bbb Q}[{\Bbb Z}]$-submodule of $Y$. 
Then we have an exact sequence. $$1\to Y_0\to G\to 
G/Y_0\to 1.$$ By induction and by $(2)$ and $(3)$ of 
Lemma \ref{inverse} it is enough to prove the $\ficwv$ 
for $Y_0\rtimes {\Bbb Z}$, which is true by hypothesis 
since $Y_0\rtimes {\Bbb Z}$ is a closely crystallographic group.

Next assume that $n>0$. Then it follows that $G/Y$ is isomorphic 
to ${\Bbb Q}^n\wr {\Bbb Z}$ for which $(1)$  
shows that the $\ficwv$ is true. Now again we apply 
$(2)$ and $(3)$ of Lemma \ref{inverse} to the homomorphism 
$G\to G/Y$ and hence we need only to show the $\ficwv$ when 
$G/Y$ is infinite cyclic. But this is again the case 
$n=0$ treated above.\end{proof}

\begin{proof}[Proof of Theorem \ref{wreath}]
The proof is immediate from Theorem \ref{general}, Lemmas \ref{pvc}, 
\ref{tvc} and \ref{lvc}.\end{proof}

\begin{proof}[Proof of Theorem \ref{ic}] Let $G$ be a closely 
crystallographic group. Recall that then $G$ is nearly 
crystallographic. Since nearly crystallographic 
groups are linear (see the paragraph after [\cite{FL}, Definition]) 
the following hold by [\cite{FL}, Theorem 1.1] and the discussion 
following it.

$$Wh(G)=\tilde K_0({\Bbb Z}[G])=K_i({\Bbb Z}[G])=0$$ 

for all negative integers $i$.

Let $G=A\rtimes {\Bbb Z}$ where $A$ is torsion free abelian.
Since $A$ is a direct limit of its finitely generated 
subgroups and since the functors in the above display commute 
with direct limit, the display also holds if we replace $G$ by $A$.

As the Whitehead groups of the groups ($G$ and $A$) we are considering 
vanish the surgery $L$-groups of these groups with different 
decorations coincide. Therefore we denote the surgery groups 
by the simple notation $L_n(-)$.

Let us first show that the non-connective assembly map 
in $L$-theory is an isomorphism for $G$. That is $$H_n(K(G, 1), 
{\ul {\Bbb L}}_0)\to L_n({\Bbb Z}[G])$$ is an isomorphism. 
 
Since the isomorphism of the above assembly map is invariant under 
taking direct limit of groups and since the map is an isomorphism 
for finitely generated free abelian groups (\cite{FH}), 
it follows that $H_n(K(A, 1),     
{\ul {\Bbb L}}_0)\to L_n({\Bbb Z}[A])$ is also an isomorphism. 

Let us now recall the following Ranicki's Mayer-Vietoris 
type exact sequence of surgery groups (\cite{Ran}) for $G$. 
$$\cdots\to L_{n+1}(G)\to L_n(A)\to L_n(A)\to L_n(G)\to\cdots.$$

There is a similar exact sequence for the homology theory 
$H_n(-, {\ul {\Bbb L}}_0)$. Now, since the assembly map is natural, a 
five lemma argument imply that $H_n(K(G, 1),     
{\ul {\Bbb L}}_0)\to L_n({\Bbb Z}[G])$ is an isomorphism. 

Next, the above $K$-theoretic vanishing result and an application 
of the Rothenberg's exact sequence (see the proof of Corollary \ref{borel}) 
implies that the 
IC$^{L^i}(G)$ is satisfied 
for $i=\l -\infty\r, h$ or $s$.\end{proof}

\begin{proof}[Proof of Theorem \ref{residually1}]
{\bf Proof of A.} $(A)$ is an immediate consequence of [\cite{R1}, 
$(1)$ of Proposition 2.2] and Lemmas \ref{pvc}, \ref{tvc} and 
\ref{lvc}. Recall that in [\cite{R1},
$(1)$ of Proposition 2.2] we assumed that the equivariant homology theory 
should be continuous when the graph of groups is infinite. But there 
this continuity assumption was used to get Lemma \ref{lvc} for the 
corresponding homology theory. Since we have noted that 
for $L$-theory Lemma \ref{lvc} is true we do not need this assumption 
here. 

{\bf Proof of B(i).} $B(i)$  
follows using Lemma \ref{pvc}, \ref{tvc} and \ref{lvc} and the following 
Proposition \ref{next}.

{\bf Proof of B(ii).} At first recall that virtually polycyclic groups 
are residually finite. And the $\flvu$ is true for virtually polycyclic 
groups by [\cite{R3}, Theorem 1.1 and (iv) of Theorem 1.3]. Also note that 
by the same result IA$(K)$ is true in the $L$-theory case for any 
virtually polycyclic group $K$. This completes the proof using $B(i)$.\end{proof}

\begin{prop} \label{next} Assume the same hypotheses as in $B(i)$ of 
Theorem \ref{residually1} replacing 
$L$ by an arbitrary equivariant homology theory and in addition assume that
$_{wt}{\cal T} _{\cal {VC}}$,
${\cal P} _{\cal {VC}}$ and
${\cal L} _{\cal {VC}}$ are satisfied. Then the $\ficwv$ 
is true for $\pi_1({\cal G})$.\end{prop}

\begin{proof} We need to apply $(3)$ of Lemma \ref{inverse} 
to the homomorphism $f:\pi_1({\cal G})\to Q$. 
Let $C$ be an infinite cyclic subgroup of $Q$. Then $f^{-1}(C)$ 
is isomorphic to the fundamental group of a graph of groups 
whose edge groups are finite and vertex groups are subgroups of 
groups of the form $K\rtimes {\Bbb Z}$ where $K$ is finitely 
generated and residually 
finite and by IA$(K)$ $K\rtimes {\Bbb Z}$ satisfies the $\ficwv$.
The proof will be completed by [\cite{R1}, $(1)$ of Proposition 2.2] 
once we show that $K\rtimes {\Bbb Z}$ is residually finite. We apply 
[\cite{R1}, Lemma 4.2]. That is, we have to show that 
given any finite index subgroup $K'$ of $K$ there is a (finite index)  
subgroup $K''$ of $K'$ which is normal in $K\rtimes {\Bbb Z}$ and the 
quotient $(K\rtimes {\Bbb Z})/K''$ is residually finite. Since 
$K$ is finitely generated we can find a finite index characteristic 
subgroup $K''$ (and hence normal in $K\rtimes {\Bbb Z}$) 
of $K$ contained in $K'$. Then 
$(K\rtimes {\Bbb Z})/K''$ is residually finite since it is 
virtually cyclic by [\cite{R1}, Lemma 6.1].

This completes the proof of the Proposition.\end{proof}

\begin{proof}[Proof of Theorem \ref{graph-group}] Proofs 
of $(A)$ and $(B)$ follow using the following.

$\bullet$ Finitely generated nilpotent groups are virtually polycyclic.

$\bullet$ Lemmas \ref{pvc}, \ref{tvc} and \ref{lvc}.

$\bullet$ [\cite{R1}, $(3)$ of Proposition 2.2], which says that 
the statements $(A)$ and $(B)$ are true for general equivariant 
homology theories ${\cal H}^?_*$ if ${\cal H}^?_*$ is continuous and 
${\cal P} _{\cal {VC}}$ and 
$_{wt}{\cal T} _{\cal {VC}}$ are satisfied. In the 
proof of [\cite{R1}, $(3)$ of Proposition 2.2] we needed the fact that 
${\cal L} _{\cal {VC}}$ is satisfied which is implied 
by the hypothesis that ${\cal H}^?_*$ is continuous 
(see [\cite{R1}, Proposition 5.1]). This completes the argument using the 
previous item.    

The proof of $(C)$ follows from the following. At first 
assume that the graph of groups is finite which we can 
by Lemma \ref{lvc}.

$\bullet$ $\pi_1({\cal G})\simeq \pi_1({\cal H})$ where 
$\cal H$ is a graph of groups whose edge groups are finite and 
each vertex group is either virtually cyclic or fundamental group  
of a tree of infinite virtually cyclic abelian groups. 
See [\cite{R1}, Lemma 3.1].

$\bullet$ The vertex groups of $\cal H$ are residually finite. See 
[\cite{R1}, Lemma 4.4].

$\bullet$ The vertex groups satisfy $\flvu$. Use $(1)$ and 
[\cite{R1}, Lemma 3.5] which implies that the graph of groups 
$\cal G$ has the intersection property.  

For the proof of $(D)$ we need the following.

$\bullet$ Lemmas \ref{pvc}, \ref{tvc} and \ref{lvc}.

$\bullet$ [\cite{R1}, $2(i)$ of Proposition 2.3]. Here note that 
for the proof of [\cite{R1}, $2(i)$ of Proposition 2.3] we 
needed that the $\ficwv$ is true for 
${\Bbb Z}^n\rtimes {\Bbb Z}$ for all $n$, which 
is the case for the $\flvu$ by [\cite{R3}, Theorem 1.1 and $(iv)$ of 
Theorem 1.3].

This completes the proof.\end{proof}

\begin{proof}[Proof of Theorem \ref{reduction1}] Let $\cal G$ be a 
graph of groups. If $\cal G$ is a tree then there is nothing to prove. 
So assume that it is not a tree. Then there is a surjective homomorphism 
$f:\pi_1({\cal G})\to F$ where $F$ is a countable free group. And the 
kernel of $f$ is a tree of groups (the universal covering graph of groups 
of $\cal G$). Now using the hypothesis, Lemma \ref{abelian1} and 
$(2)$ and $(3)$ of Lemma \ref{inverse} we complete the proof of $(A)$. 

For the proof of $(B)$ we just need to note that the universal covering 
graph of groups of $\cal G$ is a tree of groups whose class of 
vertex and edge groups is same as that of $\cal G$.\end{proof}

\begin{proof}[Proof of Theorem \ref{abelian}]
By $(B)$ of Theorem \ref{reduction1} we can assume that the graph of 
groups is a tree of finitely generated abelian groups. Next, 
by [\cite{R1}, Lemma 3.3] there is a surjective homomorphism 
$p:\pi_1({\cal G})\to H_1(\pi_1({\cal G}), {\Bbb Z})$ so that the 
restriction of $p$ to any vertex group has trivial kernel. This 
implies that the kernel of $p$ acts on a tree with trivial 
stabilizers and hence it is a free group. Now using $(2)$ and $(3)$ of 
Lemma \ref{inverse} and Lemma \ref{abelian1} we complete the 
proof.\end{proof}

\begin{proof}[Proof of Theorem \ref{residually2}] The proof 
of $(A)$ follows from Lemmas \ref{pvc}, \ref{tvc}, \ref{lvc} and 
[\cite{R1}, $(1)$ of Proposition 2.2]. The proof of $(B)$ is routine 
using  $(A)$ and $(2)$ and $(3)$ of Lemma \ref{inverse}. The only 
fact we need to mention is that a virtually residually finite group 
is residually finite.\end{proof}

\section{Some special cases}
In this section we deduce some results for the following simple 
cases of graphs of groups. This is contrary to the situation 
of ascending HNN extension for which the Fibered 
Isomorphism Conjecture is still not proved. The simplest 
case is the groups ${\Bbb Z}*_{\Bbb Z}$ where the 
two inclusions ${\Bbb Z}\to {\Bbb Z}$ are identity and 
multiplication by $2$. Note here that ${\Bbb Z}*_{\Bbb Z}=
{\Bbb Z}[\frac{1}{2}]\rtimes {\Bbb Z}$. See Remark \ref{miller}.

\begin{prop} \label{hnn} Let $G$ and $A$ be two 
groups. Let $i_j:A\to G$ 
be two injective homomorphism for $j=1,2$. Assume that there 
exist an automorphism $\alpha:G\to G$ with the property that 
$\alpha (i_1(a))=i_2(a)$ for all $a\in A$. Then the $\flf$ 
is satisfied for the $HNN$-extension $G*_A$ (defined by the 
two homomorphism $i_1$ and $i_2$) provided $G$ also 
satisfies the $\flf$.\end{prop}

\begin{prop} \label{generalized} 
Let $G_1$ and $G_2$ be two groups. Let $A$ be a 
group with two injective homomorphism $i_j:A\to G_j$ for 
$j=1,2$. Assume that there is an isomorphism $\alpha:G_1\to G_2$ 
with the property that $\alpha (i_1(a))=i_2(a)$ for 
each $a\in A$. Then the $\flf$ 
is satisfied for the generalized free product 
$G_1*_AG_2$ (defined by the 
two homomorphism $i_1$ and $i_2$) provided $G_1$ (or $G_2$) 
also satisfies the $\flf$.\end{prop} 

The following is an immediate corollary of 
Proposition \ref{generalized}.

\begin{cor} Let $M$ and $P$ be two compact manifold with nonempty 
connected $\pi_1$-injective boundary and let $f:M\to P$ be a homotopy 
equivalence so that $f|_{\p M}:{\p M}\to {\p P}$ is a 
homeomorphism. Then 
the $\flf$ is true for $\pi_1(M\cup_{\p} P)$ if the 
$\flf$ is true for $\pi_1(M)$. Here $M\cup_{\p}P$ is the union 
of $M$ and $P$ glued along the boundary via the map $f$.\end{cor}

\begin{proof}[Proof of Proposition \ref{hnn}] At first note that there 
is an obvious surjective homomorphism $f:G*_A\to G\rtimes_{\alpha}\l t\r$. 
Using $(2)$ of Lemma \ref{inverse} it follows that the $\flf$ 
is true for $G\rtimes \l t\r$ for any action of $\l t\r$ over 
$G$. Now note that the group $G*_A$ acts on a tree with vertex 
groups conjugates of $G$ and edge groups conjugates of $A$. And 
also the restrictions of $f$ to the vertex groups are injective. 
Therefore $(B)$ of Theorem \ref{residually2} completes the proof.\end{proof}

\begin{proof}[Proof of Proposition \ref{generalized}] 
Let us consider the free product $G=G_1*G_2$. Then 
there are two inclusions $j_1$ and $j_2$ from $A$ to $G$ defined by 
$i_1$ and $i_2$. And there is an isomorphism 
$\tilde \alpha :G\to G$ defined by $\alpha$ so that 
$\tilde \alpha(j_1(a))=j_2(a)$. Next note that there is an 
embedding $G_1*_AG_2\to G*_A$ where $G_1*_AG_2$ is defined 
with respect to $i_1$ and $i_2$ and $G*_A$ is defined 
with respect to $j_1$ and $j_2$. Hence 
by Lemma \ref{heredi} it is enough to prove the $\flf$ for 
$G*_A$. Since by Lemmas \ref{pvc}, \ref{tvc} and \ref{claim} 
the $\flf$ is true for $G$ we are done using Proposition \ref{hnn}.\end{proof}

\begin{rem}{\rm The Propositions \ref{hnn} and \ref{generalized} 
can be proven for arbitrary homology theories and with respect to 
the class $\cal {FIN}$ of finite groups if we add the extra assumptions that 
$_{wt}{\cal T} _{\cal {FIN}}$ and 
${\cal L} _{\cal {FIN}}$ are satisfied. We have 
already mentioned in the introduction that $_{wt}{\cal T} _{\cal {FIN}}$ 
in the $K$-theory case is still not known.}\end{rem} 
 
\section{Some consequences}
The following are some of the well-known consequences of the Isomorphism 
Conjecture.

\begin{cor} \label{whitehead} If $\G$ is a torsion free group for 
which the Fibered Isomorphism 
Conjecture in pseudoisotopy theory is true, then the following holds. 

The Whitehead group $Wh(\G)$, the lower $K$-groups 
$K_{-i}({\Bbb Z}\G)$ for $i\geq 1$ and 
the reduced projective class group $\tilde K_0({\Bbb Z}\G)$  
 vanish.\end{cor}
 
\begin{cor}\label{novikov} In addition to the hypothesis of the previous 
corollary, if the Isomorphism 
Conjecture in $L^{\l -\infty \r}$-theory is also true for the 
group $\G$ then the following holds. 

The following assembly map is an isomorphism 
for all $n$ and for $j=\l -\infty\r, h$ and $s$. $$H_n(B\G; {\bf
L}^{j}({\Bbb Z}))\to L^{j}_n({\Bbb
Z}\G).$$\end{cor} 

Note that the above two Corollaries give further evidence to the 
Whitehead Conjecture and the integral Novikov Conjecture 
respectively. The Whitehead Conjecture says that the 
Whitehead group of any torsion free group vanishes. And 
the integral Novikov Conjecture says that the above 
assembly map is split injective for torsion free groups.

Corollaries \ref{whitehead} and \ref{novikov} together imply 
the following.

\begin{cor}\label{borel} ({\bf Generalized Borel Conjecture}) 
Let $M$ be a closed aspherical manifold with $\pi_1(M)$ 
isomorphic to $G$ where $G$ satisfies the Fibered Isomorphism Conjecture 
for the pseudoisotopy and the $L$-theory cases. Then 
$M\t {\Bbb D}^k$ satisfies the Borel Conjecture for 
dim$(M)+k\geq 5$. That is, if $f:N\to M\t {\Bbb D}^k$ is a 
homotopy equivalence from another compact manifold so that 
$f|_{\p N}:{\p N}\to M\t {\Bbb S}^{k-1}$ is a homeomorphism, 
then $f$ is homotopic, relative to boundary, to a 
homeomorphism.\end{cor}

Finally we recall that, in our earlier works together with the 
present article we have 
proved the Fibered Isomorphism Conjecture both for the pseudoisotopy 
and for the $L^{-\l \infty \r}$-theory for a large class of groups. 

Below we sketch the proofs of the above corollaries.

The arguments for the proofs of Corollaries \ref{whitehead} and 
\ref{novikov} and Theorem \ref{borel} are already 
there in the literature. We 
briefly recall the proofs and then refer to the 
original sources. 

\begin{proof}[Proof of Corollary \ref{whitehead}] This is 
a consequence of the Fibered Isomorphism Conjecture 
in stable topological pseudoisotopy theory. 
See [\cite{FJ}, 1.6.5] for details. Also see [\cite{FR}, 
Theorem D].\end{proof}

\begin{proof}[Proof of Corollary \ref{novikov}]  
The Isomorphism Conjecture in $L^{\l-\infty\r}$-
theory for torsion free groups implies the isomorphism 
of the assembly map $$H_n(B\G; {\bf
L}^{\l j\r}({\Bbb Z}))\to L^{\l j\r}_n({\Bbb
Z}\G)$$ for $j=-\infty$. See [\cite{FJ}, 1.6.1] 
for details. 

Now recall the following Rothenberg exact sequence. 
$$\cdots\to L^{\l i+1\r}_n(R)\to L^{\l i\r}_n(R)\to \hat{H}^n({\Bbb Z}/2; 
\tilde K_i(R))$$$$\to L^{\l i+1\r}_{n-1}(R)\to L^{\l i\r}_{n-1}(R)\to\cdots .$$
Where $R={\Bbb Z}\G$ and $i\leq 1$. Recall that $L_n^{\l1\r}=L_n^h$ and 
$L_n^{\l-\infty\r}$ is the limit of $L_n^{\l i\r}$. 
Now using Corollary \ref{whitehead} 
and by a Five Lemma argument we get the isomorphism of the assembly 
map $$H_n(B\G; {\bf
L}^{h}({\Bbb Z}))\to L^{h}_n({\Bbb
Z}\G).$$ Using a similar Rothenberg exact 
sequence which connects the surgery groups 
with $h$ and $s$ decorations and the Tate cohomology 
which appears is with coefficient in the 
Whitehead group, one can show the following isomorphism. 
$$H_n(B\G; {\bf
L}^{s}({\Bbb Z}))\to L^{s}_n({\Bbb
Z}\G).$$ See [\cite{LR}, Section 1.5] for 
details and for other related features.\end{proof}

\begin{proof}[Proof of Corollary \ref{borel}] Let us first recall 
the surgery exact sequence. This sequence 
is for simple homotopy types and for the surgery groups 
with the decoration `$s$'. Since the 
Whitehead group of the group $G$ in the present situation 
vanishes, there is no difference between 
`$s$' and `$h$' and therefore we do not use any decoration.
 
$$\cdots\to H_n(X; \ul {\Bbb L}_0)\to L_n(\pi_1(X))\to 
{\cal S}_n(X)\to H_{n-1}(X;\ul {\Bbb L}_0)\to\cdots.$$

Where ${\cal S}_*(-)$ is the total surgery obstruction groups 
of Ranicki and $\ul {\Bbb L}_0$ is 
a $1$-connective $\Omega$-spectrum with $0$-space homotopically 
equivalent to $G/TOP$. If $X$ is a compact 
$n$-dimensional manifold ($n\geq 5$) 
then the following part of the 
above surgery exact sequence

$$\cdots\to {\cal S}_{n+2}(X) \to H_{n+1}(X; \ul 
{\Bbb L}_0)\to L_{n+1}(\pi_1(X))$$$$\to
{\cal S}_{n+1}(X)\to H_n(X;\ul {\Bbb L}_0)\to L_n(\pi_1(X))$$
 
is identified with the original surgery exact sequence

$$\cdots\to {\cal S}^{\text{Top}}(X\t {\Bbb D}^1, \p (X\t {\Bbb D}^1))\to 
[X\t {\Bbb D}^1, \p (X\t {\Bbb D}^1);G/TOP, *]$$$$\to L_{n+1}(\pi_1(X))\to 
{\cal S}^{\text{Top}}(X)\to [X; G/TOP]\to L_n(\pi_1(X)).$$ 

In particular,  
for an $n$-dimensional closed manifold $X$ there is the following 
identification. 
$${\cal S}_{n+k+1}(X)={\cal S}^{\text{Top}}(X\t {\Bbb D}^k, 
\p (X\t {\Bbb D}^k)).$$

Here ${\cal S}^{\text{Top}}(P, \p P)$ denotes the structure 
set of a compact manifold $P$. When
$Wh(\pi_1(P))=(1)$ and dim$(P)\geq 5$ 
(which is the case in the present situation) the structure set 
can be defined in the following simpler way. 
${\cal S}^{\text{Top}}(P, \p P)$ 
is the set of all equivalence classes of 
homotopy equivalences 
$f:(N, \p N)\to (P, \p P)$ from 
compact manifolds 
$(N, \p N)$ so that $f|_{\p N}:\p N\to \p P$ 
is a homeomorphism. Here two such maps 
$f_i:(N_i, \p N_i)\to (P, \p P)$ for $i=1,2$ are 
said to be equivalent if there is a homeomorphism 
$h:(N_1, \p N_1)\to (N_2, \p N_2)$ 
so that $f_2\circ h$ is homotopic to $f_1$ relative 
to boundary, that is during the homotopy the 
map on the boundary is constant. 

Next, there is a homomorphism $H_k(X; \ul {\Bbb L}_0)\to 
H_k(X; {\bf L}({\Bbb Z}))$ which is an isomorphism 
for $k>n$ and is injective for $k=n$.

Now using the fact that $M$ is aspherical and applying  
Corollaries \ref{whitehead} and \ref{novikov} 
we see that ${\cal S}^{\text{Top}}(M\t {\Bbb D}^k)$ contains 
only one element for $n+k\geq 5$. 
This completes the proof of Corollary \ref{borel}.

For some more details with related references see 
[\cite{LR}, Theorem 1.28] or [\cite{FJ}, 1.6.3].\end{proof}

\begin{rem}\label{finalrem}{\rm In view of the footnote in  
[\cite{R1}, introduction] we finally remark that [\cite{FJ}, Remark 2.1.3]  
is used in this paper in the following statements: $B(ii)$ of 
Theorem 1.3; $A$ and $B$ of Theorem 1.4; $D$ of Theorem 1.4 
when the vertex groups of any component subgraph has rank $>1$. 
The work on completing the proof 
of [\cite{FJ}, Remark 2.1.3] is in \cite{BFL}.}\end{rem}

\noindent
{\bf Acknowledgement.} I would like to thank F.T. Farrell for some 
helpful e-mail communications. Also I am grateful to Wolfgang L\"{u}ck for 
pointing out an error in a preprint which initiated some of the 
results in this  paper.

\newpage
\bibliographystyle{plain}
\ifx\undefined\bysame
\newcommand{\bysame}{\leavevmode\hbox to3em{\hrulefill}\,}
\fi

\end{document}